\documentclass[12pt]{article}
\usepackage{latexsym, amssymb, upref, amsmath, amsthm, url, color}
\newtheorem{theorem}
{Theorem}

\newtheorem{lemma}[theorem]{Lemma}

\newtheorem{proposition}[theorem]{Proposition}

\newtheorem*{conjecture}{Conjecture}

\theoremstyle{definition}

\theoremstyle{definition}

\theoremstyle{definition}

\numberwithin{equation}
{section}
\renewcommand{\max}{\mathrm{max}\,}

\newcommand{\ex}{\mathrm{ex}\,}
\newcommand{\dg}{\mathrm{dg}\,}
\newcommand{\downdeg}{\overleftarrow{\Delta}\,}

\title{Completing Partial Packings of Bipartite Graphs}

\author{Zolt\'an F\"uredi$^{{\rm a}}$\\
\and Ago-Erik Riet$^{{\rm b}}$\\
\and Mykhaylo Tyomkyn$^{{\rm c}}$\\
\small $^{\rm a}${\it Department of Mathematics
University of Illinois at Urbana-Champaign}\\
[-0.8ex]
\small{\it Urbana, IL 61801, USA}\\
[-0.8ex]
\small{\it and}\\
[-0.8ex]
\small{\it R\'enyi Institute of the Hungarian Academy of Sciences}\\
[-0.8ex]
\small{\it Budapest, P.O.Box 127, Hungary, H-1364}\\
\small $^{\rm b}${\it University of Memphis, Department of Mathematical
Sciences}\\
[-0.8ex]
\small{\it Memphis, TN 38152-3240, USA}\\
\small $^{\rm c}${\it University of Cambridge, Department of Pure Mathematics and Mathematical Statistics}\\
[-0.8ex]
\small{\it Centre for Mathematical Sciences, Wilberforce Road}\\
[-0.8ex]
\small{\it Cambridge, CB3 0WB, England}\\
}

\begin{document}
\maketitle

\begin{abstract}
Given a bipartite graph $H$ and an integer $n$, let $f(n;H)$ be the
smallest integer such that any set of edge disjoint copies of $H$
on $n$ vertices can be extended to an $H$-design on at most
$n+f(n;H)$ vertices. We establish tight bounds for the growth of $f(n;H)$ as $n\rightarrow \infty$. In particular, we prove the conjecture of F\"uredi and Lehel \cite{FuLe} that $f(n;H) = o(n)$. This settles a long-standing open problem.
\end{abstract}

\section{Introduction}

Let $H$ be a simple graph.  A \emph{partial $H$-packing of order $n$}, or simply \emph{$H$-packing}, is a set $\mathcal{P} := \{H_1,H_2,\ldots,H_m\}$ of edge-disjoint copies of $H$ whose union forms a simple graph on $n$ vertices. We say that an $H$-packing of order $n$ is \emph{complete} or an \emph{$H$-design} if the edge sets of $H_i,~i=1,\ldots,m$ partition the edge set of the complete graph on $n$ vertices. More generally, we say that a graph $G$ can be \emph{edge-decomposed} into copies of $H$ if $G$ is the union of some $H$-packing.

A long-standing problem in design theory is to find a way of completing an $H$-packing into an $H$-design of a larger size, using as few as possible new vertices. We define $f(n;H)$ to be the smallest integer such that any $H$-packing on $n$ vertices, can be extended to an $H$-design on at most
$n+f(n;H)$ vertices.

Many bounds of the type, $f(n;H)\le c(H)n$ have been proved for various graphs $H$ by explicit constructions. A (by no means complete) list of references includes Hoffman, K\" u\c c\" uk\c cif\c ci, Lindner, Roger, Stinson \cite{HoLiRo}, \cite{KuLiRo}, \cite{Li}, \cite{Li1}, \cite{LiRo}, \cite{LiRoSt}, Jenkins \cite{Je}, Bryant, Khodkar and El-Zanati \cite{BrKhEl}. See also F\"uredi and Lehel \cite{FuLe} for a survey of their results. 

Hilton and Lindner~\cite{HiLi} were the first to prove a sub-linear bound on $f(n;H)$ for a particular $H$. More precisely, they showed that a $C_4$-packing can be completed by adding $O(n^{3/4})$ new vertices. F\"uredi and Lehel \cite{FuLe} found the right order of magnitude for $f(n;C_4)$ by proving that $f(n;C_4)=\Theta(\sqrt{n})$. They conjectured that for any bipartite graph $H$ the packing can be completed by adding $o(n)$ new vertices. Our aim in this article is to give a proof of their conjecture. 

\begin{theorem} \label{ourthm}
For every bipartite graph $H$ there is a function $f(n;H) = o(n)$ such that every $H$-packing of order $n$ can be completed to an $H$-design on at most $n+f(n;H)$ vertices.
\end{theorem}

In fact we determine the asymptotical growth of the function $f(n;H)$ exactly. We say that a (not necessarily bipartite) graph $H$ is \emph{matching-friendly} if its vertex set $V(H)$ can be partitioned into $V_1$ and $V_2$ such that $V_2$ is an independent set of vertices and the induced graph $H[V_1]$ consists of a non-empty matching and a set of isolated vertices. For example, $C_4$ is not matching-friendly, but every other cycle is. The choice of the name `matching-friendly' should become clear in the course of the proof. 

\begin{theorem}\label{ourrealthm}
If $H$ is matching-friendly, then \[f(n;H)=\Theta(\ex(n,H)/n).\] If $H$ is not matching-friendly, then \[f(n;H)=\Theta\left(\max\left\{\ex(n,H)/n, \sqrt{n}\right\}\right).\]
\end{theorem}

Here $\ex(n,H)$ stands for the extremal number of $H$, see next section for its definition.

Theorem $\ref{ourrealthm}$ applies to all, not just bipartite graphs $H$. However if $H$ is not bipartite, it just states that $f(n;H)=\Theta(n)$. This is rather easy to deduce: take a packing $\mathcal{P}_n$, whose union consists of two complete graphs on $n/2$ vertices each. Such a packing exists for infinitely many values of $n$ by Wilson's theorem. It is not hard to check that $\mathcal{P}_n$ needs $\Omega(n)$ vertices in order to be extended to an $H$-design. On the other hand, every $H$-packing can be extended to an $H$-design by adding $O(n)$ new vertices; this is a consequence of Gustavsson's theorem, to be stated in Section \ref{decomp}. 

Thus from now on we shall assume that $H$ is bipartite. Note that Theorem \ref{ourrealthm} implies Theorem \ref{ourthm}.

\section{Notation and basic Tools}\label{tools}

As usual, we write $|G|$, $e(G)$, $\delta(G)$ and $\Delta(G)$ for the the number of vertices, number of edges, minimum degree and maximum degree of a graph $G$. These quantities will also be used for multigraphs and (multi)-hypergraphs. Let $N(v)$ be the neighbourhood of $v$, excluding $v$. 

Let $K_n$ and $K_{m,n}$ denote respectively the complete graph on $n$ vertices and the complete bipartite graph with bipartition classes of size respectively $m$ and $n$. The graph $K_{1,k}$ is also called a \emph{$k$-star}. It has a \emph{central vertex} of degree $k$ and $k$ \emph{endvertices} or \emph{leaves} of degree $1$.

The \emph{degeneracy} of $G$ is $\dg(G):=\max(\delta(G'))$, where the maximum is taken over all induced non-empty subgraphs $G'$ of $G$. Suppose that the vertices of $G$ are numbered $v_1,v_2,\ldots,v_n$, starting the numbering from $v_n$ backwards, so that $v_i$ is a minimum degree vertex of $G_{(i)}:=G[v_1,\ldots,v_i]$, the subgraph of $G$ induced by the vertices $v_1$ through $v_i$, for every $i=1,2,\ldots,n$. It is easy to see that $\dg(G)=\max{\delta(G_{(i)})}$. 

A \emph{transversal} of a graph $G$ is a subset $U$ of its vertices such that every edge of $G$ has at least one endpoint in $U$. In other words, transversals are complements of independent sets. The \emph{transversal number} $\tau(G)$ is the size of the smallest transversal of the graph $G$.

A graph $G$ not containing $H$ as a (not necessarily induced) subgraph is called \emph{$H$-free}. Let us denote by $\ex(n,H)$ the \emph{extremal number} for $H$, i.e.~the maximum number of edges of an $H$-free graph on $n$ vertices. More generally, let $\ex(G,H)$ be the maximum number of edges in an $H$-free subgraph of $G$. Then $\ex(n,H)=\ex(K_n,H)$. Also, if $F \subset H$ then $\ex(n,F) \leq \ex(n,H)$.

In our proof of Theorem \ref{ourrealthm} we shall use the following crude bound on \emph{symmetric Zarankiewicz numbers} $z=z(m,n,s,s)= \ex(K_{m,n},K_{s,s})$, see for instance \cite{Bo}.

\begin{theorem}\label{zarankievicz}
For all $m, n \geq s$, and $s \geq 1$ we have 
$$z(m,n,s,s)\leq 2nm^{1-1/s}+sm.$$
\end{theorem}

It is a well-known fact that $z(n,n,s,s)\ge 2 \ex(n, K_{s,s})$, see \cite{Bo}. Since every bipartite graph $H$ is a subgraph of $K_{s,s}$ for some $s$, it follows that an $H$-free graph $G$ on $n$ vertices has at most $c(H)n^{2-\epsilon(H)}$ edges, where $\epsilon = \epsilon(H)$ is a small positive number. Therefore $\delta(G)\le cn^{1-\epsilon}$. Furthermore, since a subgraph of an $H$-free graph is also $H$-free, we may conclude that $\dg(G) \leq cn^{1-\epsilon}$. A more careful estimate on the degeneracy of an $H$-free graph is given by the following lemma.

\begin{lemma}\label{degenlemma}
For every graph $H$ there is a constant $C_H$ such that for any $H$-free graph $G$ holds \[\dg(G)\leq C_H \ex(n,H)/n.\]
\end{lemma} 

In other words, every $H$-free graph $G$ of order $m\leq n$ has a vertex of degree at most $C_H \ex(n,H)/n$. Suppose first that $H$ is connected. For $m>n/2$ the graph $G$ has average degree at most $2\ex(m,H)/m \leq 4\ex(n,H)/n$. On the other hand, considering the disjoint union of two $H$-extremal graphs on $n/2$ vertices, one can see that $\ex(n,H)/n \geq \ex(n/2,H)/(n/2)$. For $m<n/2$ we can find $k$ and $l$ such that $2^k \leq n < 2^{k+1} \leq 2^l \leq m < 2^{l+1}$. We obtain \[\delta(G)\leq \frac{2\ex(m,H)}{m}\leq \frac{4\ex(2^{k+1},H)}{2^{k+1}}\leq \frac{4\ex(2^l,H)}{2^l}\leq \frac{8\ex(n,H)}{n}.\]

If $H$ is disconnected, we have to be slightly more careful, since the disjoint union of two identical $H$-free graphs $G_1$ and $G'_1$ can be no longer $H$-free. For simplicity assume that $H$ has two connected components $H_1$ and $H_2$ and that neither of them is an isolated vertex (the proof generalizes easily to the general case). The above situation can happen only if $G_1$ contains both $H_1$ and $H_2$, but any two such subgraphs have a vertex in common. Thus, by taking a fixed copy of $H_1$ in $G_1$ and deleting all edges incident with its vertices, we can make sure that the resulting graph is $H_2$-free. Doing the same with $G'_1$ we obtain an $H$-free graph on $n$ vertices with at least $2(\ex(n/2,H)-(n/2)|H|)$ edges, thus  \[\frac{\ex(n,H)}{n}\geq \frac{\ex(n/2,H)}{n/2}-\left|H\right|.\] Continuing the argument as in the case when $H$ was connected yields \[\delta(G)\leq \frac{8\ex(n,H)}{n} + 4 \left|H\right|\log_2 {n} .\] So if $H$ contains a cycle, we know \cite{Bo} that $\ex(n,H)=\Omega(n^{1+\epsilon})$, therefore the term $4 \left|H\right|\log_2 {n}$ can be neglected.

Finally, if $H$ is a forest, it is known, see e.g. \cite{Bo}, that $m/2 \leq ex(m,H) \leq c_H m$, apart from the trivial case when $H$ has only one edge. So we can take $C_H = 4c_H$. This finishes the proof of Lemma \ref{degenlemma}. Note that only for disconnected forests did the constant $C_H$ actually depend on $H$.

We shall need two basic facts about graph colouring. Their proofs can be found in any standard textbook on graph theory e.g. \cite{Bo1}. One is the fact that a graph of maximal degree $\Delta$ can be $\Delta+1$-coloured by a greedy algorithm. The other theorem we need is Vizing's theorem: a graph of maximal degree $\Delta$ can be edge-coloured using $\Delta+1$-colours or, equivalently, can be decomposed into $\Delta+1$ matchings.

\section{A Primer on Graph Decompositions}\label{decomp}

In this section we would like to state various theorems on graph decompositions that we shall use in the proof.

Let $H$ be a bipartite simple graph of order $d$ with vertices $v_1,v_2,\ldots,v_d$ and let $\deg(v_i)$ denote the degree of $v_i$. Denote $\gcd(H) = \gcd(\deg(v_1),\ldots,\deg(v_d))$. For an $H$-design of order $n$ to exist we need the following obvious conditions:
\[e(H)|\binom{n}{2}\;\; \text{and} \;\; \gcd(H) | (n-1).\] If these conditions hold we say that $n$ is \emph{$H$-divisible}. If $n$ admits an $H$-design, we call it \emph{$H$-admissible}. Wilson \cite{Wi} proved the following fundamental theorem.

\begin{theorem}\label{wilson}
There exists an integer $n_0$, depending on $H$, such that every $n>n_0$ that is $H$-divisible is also $H$-admissible.
\end{theorem}

Wilson's theorem implies that $f(n;H)$ exists for every $H$ and $n$. Indeed, the union of an $H$-packing $\mathcal{P}$ on $n$ vertices can be considered as our new `small' graph $H'$. By Theorem \ref{wilson} there exists an $H'$-design $\mathcal{P'}$ for a sufficiently large $H'$-divisible number $n'$. By decomposing each copy of $H'$ in $\mathcal{P'}$ into copies of $H$, we obtain an $H$-design on $n'$ vertices. Since for a given $n$ there are only finitely many $H$-packings on $n$ vertices, and each of them can be completed to an $H$-design as above, $f(n;H)$ is well-defined. 

More generally, let us say a graph $G$ is $H$-divisible, if all degrees of $G$ are multiples of $\gcd(H)$ and $e(H)|e(G)$.

A very deep and powerful extension of Wilson's theorem was proved by Gustavsson~\cite{Gu}.

\begin{theorem}\label{gustdigr}
For any digraph $D$ there exist $\epsilon_D > 0$ and $N_D > 0$ such that if $G$ is a digraph satisfying:
\begin{enumerate}
\item $e(G)$ is divisible by $e(D)$;
\item there exist non-negative integers $a_{ij}$ such that
$$ \sum_{v_i\in V(D)} a_{ij} d_D^+(v_i) = d_G^+(u_j), \;\;\; \sum_{v_i\in V(D)} a_{ij} d_D^-(v_i) = d_G^-(u_j) $$
for every $u_j \in V(G)$;
\item if there exists $\vec{u_1u_2} \in E(G)$ such that $\vec{u_2u_1} \not\in E(G)$ then there exists 
$\vec{v_1v_2} \in E(D)$ such that $\vec{v_2v_1} \not\in E(D)$;
\item $|V(G)| \geq N_D$;
\item $\delta^+, \delta^- > (1-\epsilon_D)|V(G)|$
\end{enumerate}
then $G$ can be written as an edge-disjoint union of copies of $D$.
\end{theorem}

Viewing simple graphs $G$ and $H$ as digraphs, by orienting each edge in both directions, the above theorem translates to

\begin{theorem}\label{gust}
For every $H$ there exist $m_0$ and $\epsilon_0$ such that every $H$-divisible graph $G$ on $m> m_0$ vertices with minimum degree at least
$(1-\epsilon_0)m$ can be edge-decomposed into copies of $H$.
\end{theorem}

In the proof of Theorem \ref{ourrealthm} we shall need the analogue of Wilson's theorem for $H$-packings into complete bipartite graphs $K_{m,n}$, in which case the obvious divisibility conditions are $$e(H)|mn\;\; \text{,} \;\; \gcd(H) | m \;\; \text{and} \;\; \gcd(H)| n.$$

\begin{theorem}\label{hagg}
Let $H$ be a bipartite graph. There exists an integer $n_0$, depending on $H$, such that every $H$-divisible $K_{m,n}$ with $m, n>n_0$ can be edge-decomposed into copies of $H$.
\end{theorem}

This was proved by H\"aggkvist~\cite{Ha} for the case when $H$ is regular, $m=n$, and under stronger divisibility assumptions. However, H\"aggkvist's proof was before Gustavsson's theorem. With Theorem \ref{gustdigr} at our disposal, we can give a proof of Theorem \ref{hagg}. While it is almost certain that its statement has been well-known, we could not find any explicit reference. Thus, we shall give a proof sketch, skipping some technical details.

\begin{proof}
First suppose that $m=n$. The graph $K_{n,n}$ on vertices $\{1,...,n\}$ and $\{1',...,n'\}$ can be thought of as a directed graph with loops on $\{1,...n\}$ by replacing each edge $ab'$ with a directed edge $a$ to $b$. By embedding $H$ so that the bipartite classes of $H$ are sent to disjoint subsets of $\{1,...,n\}$ we can
regard $H$ as a directed graph $H'$ without loops. By removing $n$ copies of $H$
from $K_{n,n}$ first, where each copy has exactly one 'vertical' edge, we
reduce to the case of decomposing a dense digraph $G$ (without loops) into
copies of the digraph $H'$. Here `dense' means that we must ensure that
$\delta^\pm(G)>(1-\epsilon)n$. The packing of G can be done provided (a) $n$
is large enough; (b) the number of edges is divisible by $e(H)$; and (c) the
in- and out-degrees of any vertex of $G$ are representable as a non-negative
linear combination of the in- and out- degrees of vertices of $H'$. This last
condition should translate to the assumption than $n$ is divisible by both the
$\gcd$ of the degrees of the vertices in $A$ and the $\gcd$ of the degrees of the
vertices in $B$, where $(A,B)$ is the bipartition of $H$. (This assumes one wants
to pack all the copies of $H$ the same way round. If not, pack $H\cup H^r$ where
$H^r$ is $H$ with the bipartition reversed, and possibly remove one extra copy
of $H$ initially to ensure that $2e(H)$ divides $e(G)$. Then $n$ need only be
divisible by the $\gcd(H)$.)

So there is an integer $n'_0$ such that the theorem holds for all $K_{n,n}$ with $n>n'_0$. In fact, the same construction works for $K_{m,n}$ if $n \le m \le (1+\epsilon'(H))n$ -- remove some copies of $H$ in order to isolate $m-n$ vertices in the larger partition class, making sure that we do not reduce the degrees of the remaining vertices too much. Having done that, apply the above digraph reduction to the remaining graph, which can be viewed as a subgraph of $K_{n,n}$. Then apply Theorem \ref{gustdigr} as above.

Given $m,n\ge n_0 = (n'_0)^2$, we can partition both sets $\{1,...,m\}$ and $\{1,...,n\}$ into subsets of size about $n_0$ each and such that each complete bipartite graph $(X,Y)$ induced on two partition classes $X\subset \{1,...,m\}$ and $Y \subset \{1,...,n\}$ is $H$-divisible. Pack every such graph with copies of $H$ as described above.
 
\end{proof}

\section{Upper bound: Outline of the Proof}

In this section we would like to describe our strategy for proving the upper bound in Theorem \ref{ourrealthm}.

Consider an $H$-packing $\mathcal{P} = \{H_1,H_2,\ldots,H_m\}$ on $n$ vertices. We want to complete it to an $H$-design by adding few vertices. We consider the \emph{uncovered graph} $G_0 = (K_n)\setminus \cup_{i=1,\ldots,m} E(H_i)$ i.e.~the graph consisting of edges that are not covered by copies of $H$.

We proceed in three steps:

\textbf{Step 1}: Reducing the transversal. We add some new vertices and all possible edges from those to other vertices. Now we delete an edge-disjoint collection of copies of $H$ from the resulting graph, so that the resulting graph has a smaller transversal than the graph we started with. This step constitutes a major part of the proof of Theorem \ref{ourrealthm} and will be carried out in Sections \ref{degen} through \ref{transversal}. 

More precisely, in Section \ref{degen} we shall construct a `nice' collection of disjoint $k$-stars on the edges of any given graph $G$. This construction will be applied in Section \ref{hyper} to $G_0$ in order to construct a hypergraph $M$ with a small edge-chromatic number, related to $G_0$. Then in Section \ref{transversal} we shall use $M$ and its edge-colouring in order to extend $\mathcal{P}$ to a packing on a larger vertex set, such that the uncovered graph has a small transversal.

In Section \ref{further} we shall describe how we iterate Step 1 in order to obtain further packings with yet smaller transversals of the uncovered graphs.

\textbf{Step 2}: Decreasing the number of uncovered edges. Starting with an uncovered graph $G_1$ that has a small transversal we extend the new packing to obtain a new uncovered graph $G_2$ with very few edges. This will be established in Section \ref{decreasing}.

\textbf{Step 3}: Completing the packing. This will be done by applying Theorem \ref{gust} and Theorem \ref{hagg} in Section \ref{completing}.


\section{Degeneracy}\label{degen}

The aim of this section is to prove the Proposition \ref{stars}, which will be our main tool for reducing the transversal of the uncovered graph. We also believe that the statement of Proposition \ref{stars} is interesting in its own right; see Section \ref{outlook} for related questions.

Recall that a $k$-star is a copy of $K_{1,k}$.

\begin{proposition}\label{stars}
For every integer $k$ and a graph $G$ of degeneracy $d$ there is a maximal collection $\mathcal{C}$ of edge disjoint $k$-stars on $G$ such that each vertex of $G$ is an endvertex to at most $d+k-1$ stars in $\mathcal{C}$.
\end{proposition}

Case $k=2$ was done by F\"uredi and Lehel \cite{FuLe}. 
We are using downdegree instead of updegree since this feels more natural to us. Let us choose an ordering $v_1,v_2,\ldots,v_n$ of vertices of $G_0$ such that the \emph{(maximum) downdegree} $\downdeg(G_0)$,  defined as the maximum of the number of edges from a vertex $v_i$ to vertices $v_j,\;j<i$, over all $i=1,2,\ldots,n$, equals $d=\dg(G)$. 

Let us construct $\mathcal{C}$ as follows: take a maximal collection of edge-disjoint
$k$-stars whose central vertex is smaller in the given ordering than any of its
endvertices, and then extend it to a maximal collection of
edge-disjoint $k$-stars. Then $u\in G$ appears as an endvertex of
a star of the first kind, or as such endvertex of a star of the
second kind which is greater than its centre at most $\downdeg(G)$
times. It appears as an endvertex smaller than the centre of a star
of the second kind at most $k-1$ times since otherwise we could form a
star of the first kind with $u$ at its centre -- this is a
contradiction as we started taking stars of the second kind in a
graph containing no stars of the first kind.

It follows that $u$ can appear at most $\downdeg(G)+k-1 = d+k-1$ times as an endvertex of a star in $\mathcal{C}$, which proves Proposition \ref{stars}. Note that the maximality of $\mathcal{C}$ implies $\Delta(G\setminus \bigcup \mathcal{C})\leq k-1$.

\section{Construction of a Hypergraph and its Colouring}\label{hyper}

Recall that we are given an $H$-packing $\mathcal{P} = \{H_1,H_2,\ldots,H_m\}$ on $n$ vertices and $G_0 =(K_n)\setminus \cup_{i=1,\ldots,m} E(H_i)$ is our uncovered graph. 

In this section we shall give a construction of a certain hypergraph $M$ on a vertex set of $G_0$ along with its edge-colouring; we shall need it in order to extend $\mathcal{P}$ to a packing on a larger set of vertices, in which the uncovered graph will have a small transversal. 

First of all, we can assume without loss of generality that $G_0$ is $H$-free (by removing a maximal set of edge-disjoint copies of $H$ from $G_0$). By Lemma \ref{degenlemma} we know that $\dg(G_0)=O(\ex(n,H)/n)$.

For a fixed vertex $v$ of $H$, let $k=\deg(v)$ and $W_1 = N(v)$. Let $(U,W)$ be a bipartition of $H$ such that $v \in U$ and $W_1 \subset W$. Denote by $R$ the ratio $|W|/|W_1|$, rounded up to the nearest integer. Let $s=|U|$ and $t=|W|$ be the sizes of the bipartition classes. For convenience we can assume that $s \geq t$, perhaps choosing another $v$.

By Proposition \ref{stars} there is a collection $\mathcal{C}$ of disjoint $k$-stars on $G_0$ with the property that each vertex of $G_0$ is an endvertex to at most $\dg(G_0)+k-1$ stars in $\mathcal{C}$. Define a multi-$k$-graph ($k$-uniform hypergraph with several
edges on the same set of vertices allowed) called $M$ as follows: for
every star of $\mathcal{C}$ there is a $k$-edge containing precisely the leaves of the star. The
maximum degree $\Delta(M)$ (i.e.~the maximum number of edges
containing any given vertex) is bounded by $\dg(G_0)+k-1 \leq
c_3 * \ex(n,H)/n$, where $c_3$ is a positive constant
depending only on $H$. We shall denote edges of $M$ by $(c,e)$ where $c\in G_0$
is the centre of the respective star and $e$ is the hyperedge consisting
precisely of the leaves of the star.

Let us introduce an edge-colouring on $M$ so that each
colour class forms a vertex-disjoint collection of hyperedges. Since every hyperedge
intersects at most $k(\Delta(M)-1)$ other hyperedges, it
can be done, using at most $k(\Delta(M)-1)+1 = c_4 * \ex(n,H)/n$ colours:
let us colour greedily as many hyperedges with colour 1 as we can, then
with colour 2 and so on (again $c_4$ is a positive constant
depending only on $H$).

Split every colour class $i$ into $R = \lceil|W|/|W_1|\rceil$
(almost) equal parts $i.1$ through $i.R$. For every colour class $i$, fix a
map $\sigma_i$ which, for every $j$, takes hyperedges coloured $i.j$ to disjoint $|W_1|(R-1)$-subsets of
vertices inside the union of hyperedges coloured with one of the colours $i.l,\; l\neq j$. Note that this
mapping takes hyperedges into sets which are disjoint from the hyperedge itself.

 Now we are ready to extend $\mathcal{P}$ in order to reduce $G_0$ to a new uncovered graph $G_1$ that has a new transversal.

\section{Construction of a transversal}\label{transversal}

We shall prove that, by adding a small set of new vertices $Q$, we can use up all the edges inside $G_0$ in edge-disjoint
copies of $H$ and end up with a graph $G_1$ on the vertex set $V\cup
Q$ with no edges inside $V$ (i.e.~with transversal $Q$). 

The following construction decreases the degrees of the vertices in $V$ below $k$.

\emph{Construction 1. Covering all $k$-stars.} Write $V=V(G_0)$. Consider $v\in H$,
$k=\deg(v)$, the bipartition $H=(U,W)$ and the colouring of the multihypergraph
$M$ as before. For every colour $i.j$ add to $G_0$ a set $Q^{i.j}=q^{i.j}_1,\ldots,q^{i.j}_{|U|-1}$ of $|U|-1$ new vertices 
and place a copy of $H=(U,W)$ in the obvious way on every star $(c,e)$ of colour $i.j$ such that
$U=\{c,q_1^{i.j},\ldots,q_{|U|-1}^{i.j}\}$ and $W \subset e\cup \sigma_i(e)$ (if $|W|$ is divisible by $|W_1|$ then we have $W = e\cup \sigma_i(e)$).
Note that the sets $e\cup \sigma_i(e)$ for different hyperedges $e$ of colour $i.j$ are pairwise
disjoint and so the copies of $H$ are placed edge-disjointly. We needed $O(\ex(n,H)/n)$ new vertices.

The following construction takes care of all the edges from within $V$. 

\emph{Construction 2. Covering the remaining edges.} By Vizing's theorem, the set of remaining edges inside $V$ can be partitioned into (at most) $k$ matchings $L_1,\dots,L_k$. Consider the smallest $r$ such that $\binom{r}{2} \geq e(H)\frac{n}{2}$ and $K_r$ can be packed completely with copies of $H$. By Theorem \ref{wilson} we can pick $r=O(\sqrt{n})$. For each matching $L_i$, add to $G_0$ a set $Q^{L_i}$ of $r$ new vertices, and pack the copies of $H$ into $K_r\cup L_i$ so that the packing is almost like the complete packing of $K_r$, except with all edges in $L_i$ covered by an edge from different copies of $H$. This way we clearly pack copies of $H$ edge-disjointly. Note that $|Q^{L_i}|=O(n^{1/2})$ for every $i$, so we need $O(n^{1/2})$ new vertices for this construction.

However, if $H$ is matching-friendly, we can do much better. Recall, $H$ is matching-friendly if $V(H)$ can be partitioned into $V_1$ and $V_2$, where $V_2$ is independent and $V_1$ is `almost' independent, i.e. the $V_1$-induced subgraph of $H$ is a non-empty matching and some isolated vertices. This implies that we can cover at least one edge of an uncovered matching $L_i$ by adding $\left|V_2\right|$ new vertices such that no edge between the new vertices will be used. It follows easily that the whole $L_i$ can be covered using at most a constant number of $c(H)$ new vertices.

Let $Q = \cup_{i,j} Q^{i.j} \bigcup \cup_i Q^{L_i}$. We have constructed a graph $G_1$ on vertex set $V\cup Q$ with transversal $Q$. By removing copies of $H$, we can assume that $G_1$ is $H$-free. For the number of added vertices we have the bound $|Q|\leq c_5 * \max\left\{\ex(n,H)/n, \sqrt{n}\right\}$. If $H$ is matching-friendly, we obtain $|Q|\leq c_6 * \ex(n,H)/n$. 

\section{Further transversals}\label{further}

We can add some more vertices to $G_1$ to reduce the transversal number of the resulting graph even further. This procedure can be repeated many times.

It suffices to prove the following lemma.

\begin{lemma}\label{zaran2}
Let $G$ be an $H$-free graph on $n$ vertices, containing a transversal $Q$ of size $q=o(n)$. Then there is an ordering $v_1,\ldots,v_n$ of the vertices of $G$ such that $\downdeg(G) \le Cq^{1-\epsilon}$, where $C$ and $0 < \epsilon = \epsilon(H) < 1$ are constants depending only on $H$. In particular, $\dg(G)\le Cq^{1-\epsilon}$.
\end{lemma}

\begin{proof}
Let us write $Y=V(G)\backslash Q$ and consider the bipartite graph $G'$ with bipartition $(Y,Q)$, whose edges are the edges of $G$ having precisely one vertex in each of $Q$ and $Y$. Let $G''=G[Q]$ be the subgraph of $G$ induced by $Q$. Then the edge sets of $G'$ and $G''$ partition the edge set of $G$.

Since $G''$ is an $H$-free graph on $q$ vertices, its degeneracy is at most $c''q^{1-\epsilon}$ for a positive constant $c''$ depending only on $H$. Let us fix an ordering $u_1,u_2,\ldots,u_q$ of the vertices in $Q$ such that $\downdeg(G'')=\dg(G'')$.

Select $s$ and $t$ with $s\ge t$ such that $H\subset K_{s,t}\subset K_{s,s}$ and $s$ is chosen as small as possible. By Theorem \ref{zarankievicz} we have that
$$z(|Q|,|Y|,s, s)\leq 2|Y||Q|^{1-1/s}+s|Q|.$$
Let $\epsilon \leq 1/s$. We find that
$$\ex(K_{|Q|,|Y|},H) \leq \ex(K_{|Q|,|Y|},K_{s,s}) = z(|Q|,|Y|,s,s) \le 2|Y||Q|^{1-1/s}+s|Q|.$$
Therefore, as long as $|Y|\ge q^{1/s}$, the minimal degree in $Y$ satisfies $\delta(Y)=O(q^{1-1/s})$.

Let $v_1$ be a vertex of $Y$ of smallest possible degree in the graph $G'$, let $v_2$ be a vertex of $Y$ of minimal degree in $G'[V(G)\backslash \{v_1\}]$, take $v_3$ to be a vertex of $Y$ of minimal degree in $G'[V(G)\backslash \{v_1,v_2\}]$ and so on, until $v_r$, where $r=|Y|-q^{1/s}$. Each of those degrees is $O(q^{1-\epsilon})$, by the previous paragraph. Let $v_{r+1}, v_{r+2}\dots v_{n-q}$ be the remaining vertices in $Y$. 

Define the ordering $v_{r+1},v_{r+2}\dots v_{n-q}, u_1,u_2,\ldots,u_q,v_1,v_2,\ldots,v_r$. It follows from the construction that $\downdeg(G') \leq c'q^{1-\epsilon(H)}$.

\end{proof}

The lemma allows us to iterate the construction of Sections \ref{hyper} and \ref{transversal}. An $H$-free uncovered graph with a transversal of size $q$ has by Lemma \ref{zaran2} degeneracy $c'q^{1-\epsilon}$. Hence we can define a hypergraph as in Section \ref{hyper} and use it in order to construct a new packing as in Section \ref{transversal}. The number of new vertices needed in Construction 1 will be $O(q^{1-\epsilon})$ and in Construction 2 of Section \ref{transversal} each matching has cardinality at most $q$, so we need to add a set $Q^{L_i}$ of $O(q^{1/2})$ additional vertices for every matching $L_i$. Hence, the total number of new vertices will be at most $C(H)q^{1-\epsilon(H)}$. By construction, this set of vertices will be a transversal of the new packing, so we can just repeat the procedure, using the new transversal. We iterate as long as $Cq^{1-\epsilon} \leq q/2$, that is $q \geq C'(H)= \left(2C\right)^{1/\epsilon}$. The numbers of new vertices halves after each step, thus by adding $O(\max\left\{\ex(n,H)/n, \sqrt{n}\right\})$ new vertices, or $O(\ex(n,H)/n)$ if $H$ is matching-friendly, we can make the transversal smaller than the constant $C'(H)$. 

\section{Decreasing the number of uncovered edges}\label{decreasing}

Our next objective is to reduce the number of uncovered edges. Furthermore we shall make sure that the number of vertices in the uncovered graph is congruent $1$ modulo $e(H)$. This will be needed later for completing the packing.

Write $G_2$ for the uncovered graph with $Q \subset V(G_2)$ a transversal and $Y=V(G_2)\backslash Q$. As we know from Section \ref{further}, we may assume that $|Q| < C'(H)$. Define $g=\gcd(H)$. By adding a few new vertices to $Q$ we may also assume that $|G_2| \equiv 1 \text{ mod } e(H)$. Since $G_2$ is the complement of a partial packing and $g|e(H)$ (because $H$ is bipartite), all degrees in $G_2$ must be multiples of $g$. This implies that every vertex in $Y$ is either isolated or has at least $g$ neighbours in $Q$. We shall add a set $Z$ of new vertices of size $m|Q|^g $ in order to reduce $G_2$ to a graph $G_3$ in which every subset of vertices of $Q$ of size $g$ has at most $m$ common neighbours in $Y$ and every vertex in $Y$ has either none or at least $g$ neighbours in $Q$. That would bound the number of edges between $Y$ and $Q$ by $m|Q|^g$. In addition every vertex from $Z$ will have at most $m$ uncovered edges in $Y$ incident with it. Then $G_3$ would have at most $m|Q|^g+m|Z|+1/2(|Z|+|Q|)^2 = C''(H)$ edges.

Let $m=2n_0$, where $n_0$ a multiple of $e(H)$ that satisfies Theorem \ref{hagg} for $H$, that is any $H$-divisible complete bipartite graph with at least $n_0$ vertices in each partition class can be edge-decomposed into copies of $H$.

Let us pick a set $K = \{q_1,q_2,\ldots,q_g\}$ of some $g$ vertices in $Q$ and write $N$ for their common neighbourhood in $Y$: $N = N(q_1) \cap N(q_2) \cap \ldots \cap N(q_g) \cap Y$. If $|N| > m$, we are going to add to $G_2$ an additional set $Q^*_{q_1,\ldots,q_g} = Q^* = \{q_1^*,q_2^*,\ldots,q_m^*\}$ of $m$ vertices. If $|N|\le m$, we just pick the next $K$.

We are going to cover almost all the edges in the complete bipartite graphs $(K\cup Q^* , N)$ and $(Q^* , Y\backslash N)$. Since $|Q^*|$ and $|K\cup Q^*|$  are both divisible by $g$, to make those graphs $H$-divisible, it suffices to omit less than $e(H)$ vertices from each of the sets $N$ and $Y\backslash N$ --- so that we obtain respectively sets $N'$ and $Y'$. By Theorem \ref{hagg} it follows that both complete bipartite graphs $(K\cup Q^* , N')$ and $(Q^* , Y')$ can be packed completely with edge-disjoint copies of $H$. 

\vspace{5pt}

The uncovered graph has obtained $m$ new vertices, each of which has at most $m$ (in fact at most $2e(H)$) uncovered edges into $Y$ and the vertices in $Q^*$ have now at most $m$ common neighbours inside $Y$. Also, for each vertex in $Y$, the number of its remaining neighbours in $Q$ is a multiple of $g$.

If we repeat the procedure for all possible sets $K\subset Q$ of size $g$, we obtain the desired graph $G_3$, taking $Z$ to be the union over all $K$. Notice also that by adding $m$ vertices at a time, we make sure that $|G_3|\equiv 1 \text{ mod } e(H)$.


\section{Completing the packing}\label{completing}

We shall now apply Theorems \ref{gust} and \ref{hagg} to complete the packing. Since the uncovered graph $G_3$ has a constant number edges, the number of non-isolated vertices in it is also constant. Let $Q$ be a set of vertices of size $C_3(H)$ such that all vertices in $Y=G_3\setminus Q$ are isolated and $|Y|\equiv 0 \text { mod } e(H)$; hence also $|Y|\equiv 0 \text{ mod } g$, where $g$ is the greatest common divisor of all degrees in $H$, as before. By the construction in the previous section we may assume that $|Q|+|Y|=|G_3|\equiv 1 \text{ mod } e(H)$, thus $|Q|\equiv 1 \text{ mod } e(H)$.

\vspace{5pt}

We now apply Theorem \ref{gust} to $G_3[Q]$ to extend the packing by adding a set $X$ of few new vertices. More precisely, we pick $X$ to be set of new vertices of size $\max\left\lbrace m_0, (1/\epsilon_0)|Q|\right\rbrace$, where $m_0$ and $\epsilon_0$ are as in Theorem \ref{gust}, this is a constant of $H$. Also let $|X|\equiv |Y|+|Q|-1 \text{ mod } 2e(H)$. To complete the packing it suffices to make sure that the uncovered graph on $Q\cup X$ and the complete bipartite graph $K_{X,Y}$ are $H$-divisible.

One divisibility condition requires $|X|+|Q| \equiv 1 \text{ mod } g$ for the former graph and $|X|,|Y| \equiv 0 \text{ mod } g$ for the latter. Both conditions are satisfied since $|X|\equiv 0 \text{ mod } g$. 

The other divisibility condition requires the number of edges in each graph to be divisible by $e(H)$. This is certainly true for $K_{X,Y}$, by the choice of $Y$. So we only need to make sure that $e(H)$ divides the number of edges of the uncovered graph on $Q\cup X$, in other words \[e(H)| \left(\binom{|X|+|Q|}{2}-\binom{|Q|}{2}+e(G_3)\right).\] Since $G_3$ is the complement of an $H$-packing, we know that \[e(G_3)\equiv \binom{|Q|+|Y|}{2} \text{ mod } e(H).\] Therefore we need $e(H)$ to divide \[\binom{|X|+|Q|}{2}-\binom{|Q|}{2}+\binom{|Q|+|Y|}{2} = \binom{|X|+|Y|+|Q|}{2}-|X||Y|.\] This is true whenever $|X|\equiv |Y|+|Q|-1 \text{ mod } 2e(H)$.

Hence we can satisfy all divisibility conditions in order to apply Theorems \ref{gust} and \ref{hagg} to complete the packing. This finishes the proof of the upper bound in Theorem \ref{ourrealthm}.

\section{Lower bound}\label{lowerbound}

In this section we want to show the existence of $H$-packings that need $\Omega(\ex(n,H)/n)$ vertices in order to be completed. If $H$ is not matching-friendly, there exist also packings that need $\Omega(\sqrt{n})$ new vertices.

Let us start with the second claim. If $H$ is not matching-friendly, we need $\Omega(\sqrt{n})$ new  vertices in order to cover the edges of a complete matching $L$ on $n$ vertices. Indeed, anytime we place a copy of $H$ that covers at least one edge of $L$, we must use an edge between two new vertices (otherwise $H$ would be matching-friendly). Hence, in order to cover $n/2$ edges of $L$ we need about $\sqrt{n}$ new vertices. 

Now we have to make sure that the complement of a perfect matching \emph{is} the union of an $H$-packing for infinitely many $n$. Take two disjoint copies of $H$ and view their union $H'$ as a bipartite graph with equal partition classes, i.e. one copy of $H$ is `upside down'. Let $s$ be the size of the partition classes. By Theorem \ref{wilson}, if $n$ is sufficiently large, there is a complete packing $\mathcal{P}$ of $K_n$ with copies of $H'$. Now take two identical copies of $\mathcal{P}$, one on $\{a_1,a_2,\dots\,a_n\}$ and another on $\{b_1,b_2,\dots\,b_n\}$ and add a copy of $H'$ between $a_{i1},\dots,a_{is}$ and $b_{j1},\dots,b_{js}$ and another one between $a_{j1},\dots,a_{js}$ and $b_{i1},\dots,b_{is}$ for each copy of $H'$ in $\mathcal{P}$ between $a_{i1},\dots,a_{is}$ and $a_{j1},\dots,a_{js}$, in the obvious way. We obtain a packing on $2n$ vertices, whose union is the complement of a matching between vertices $a_i$ and $b_i$.

Now let us prove the first claim. Suppose we have found an $H$-packing $\mathcal{P}$, whose complement is an $H$-free graph with about $\ex(n,H)$ edges. In order to cover each edge of it, every copy of $H$ would use at least one out of $kn+\binom{k}{2}=(1+o(1))kn$ new edges, where $k$ is the number of new vertices. Since we need at least $\ex(n,H)/e(H)$ copies of $H$ to cover all edges of the uncovered graph, we must have $k=\Omega(\ex(n,H)/n)$.

Hence, it remains to prove that such a packing $\mathcal{P}$ exists for arbitrarily large values of $n$. Take an (extremal) $H$-free graph $\overline{G}$ on $n$ vertices with $\ex(n,H)$ edges. We would like to remove a small proportion of edges from $\overline{G}$ in order to make the complement of the remaining graph satisfy the conditions of Theorem \ref{gust}. This would ensure the existence of the desired packing.

Let us first eliminate vertices of high degree. Suppose $\overline{G}$ has $\log n$ vertices of degree at least $\epsilon_0 n$, where $\epsilon_0$ is as in Theorem \ref{gust}. Then by Theorem \ref{zarankievicz}, for a sufficiently large $n$ the bipartite graph between $m = \log n$ such vertices and the rest of $\overline{G}$ contains $K_{s,s}\supset H$, contradicting the assumption that $\overline{G}$ is $H$-free. It follows that $\overline{G}$ has less than $\log n$ vertices of degree at least $\epsilon_0 n$. Removing them, we lose at most $n \log n$ edges obtaining (unless $H$ is a forest, in which case there is nothing to prove) a new $H$-free graph $\overline{G'}$ with $(1-o(1))\ex(n,H)$ edges and no vertices of high degree.

Next we would like to remove a few more edges from $\overline{G'}$ in order to fulfill the divisibility conditions. A theorem of Pyber \cite{Py} states that a graph $F$ that has at least $n \log n * 32 r^2$ edges contains a (not necessarily spanning) $r$-regular subgraph. Let us set $r=2e(H)$. Remove edge sets of $r$-regular subgraphs $G_1 \subset \overline{G'}$, $G_2 \subset \overline{G'}\setminus G_1$ etc. until the remaining graph $\overline{G'}\setminus (G_1\cup G_2 \cup \dots \cup G_k)$ has less than $n \log n * 32 r^2$ edges. Then the graph $\overline{G''}=G_1\cup G_2 \cup \dots \cup G_k$ satisfies all conditions of Theorem \ref{gust} and still has about $\ex(n,H)$ edges, whence we obtain the desired packing $\mathcal{P}$.

\section{Outlook}\label{outlook}

What can we say about not matching-friendly bipartite graphs? If every edge of $H$ is contained in a $4$-cycle, $H$ cannot be matching-friendly. On the other hand, if $C_4 \subset H$ then $\ex(n,H)/n = \Omega(n^{1/2})$, thus being matching-friendly does not matter, as far as Theorem \ref{ourrealthm} is concerned. There are examples of bipartite $C_4$-free graphs that are not matching-friendly; take for instance $C_8$ and connect the opposite pairs of vertices by paths of length $2$. Or, alternatively, take the incidence graph of the Fano plane. However, we do not know much about the extremal numbers of such graphs, so the question is: \textbf{does `matching-friendly' ever make a difference?}

The constant $C_H$ in the proof of Lemma \ref{degenlemma} depended on $H$ only when $H$ was a disconnected forest. Is it possible to prove Lemma \ref{degenlemma} with an absolute constant, perhaps even $C_H = 1+o(1)$?

The following question was inspired by Proposition \ref{stars}. We believe it is interesting in its own right.

\begin{conjecture}
For every integer $k$ and a graph $G$ of degeneracy $d$ there is a maximal collection $\mathcal{C}$ of edge disjoint paths of length $2k$ on $G$ such that each vertex of $G$ is an endvertex to at most $c_kd$ paths in $\mathcal{C}$.
\end{conjecture}

This cannot hold for odd-length paths, as can be seen by taking, for instance paths of length $3$ and $G=K_{2,m}$, where $m$ is large. Case $k=1$ of the Conjecture is the special case of Proposition \ref{stars}; it was first proved by F\"uredi and Lehel \cite{FuLe}. We think, we can also prove the Conjecture for $k=2$. 

More generally, can $2k$-path in the statement of the Conjecture (or Proposition \ref{stars}) be replaced by a tree, in which all distances between the leaves are even? 

\section*{Acknowledgements} The authors would like to thank Jen\H o Lehel for introducing the problem and Paul Balister and B\' ela Bollob\' as for many useful discussions.

\end{document}